\documentclass[12pt,reqno]{article}
\usepackage{fullpage}
\usepackage{float}
\usepackage{upgreek}
\usepackage{booktabs,enumerate}
\usepackage{amsthm,amsmath,amssymb}
\usepackage{tikz}
\usepackage{subfigure}
\usepackage{multirow,multicol}
\usepackage[export]{adjustbox}
\definecolor{webgreen}{rgb}{0,.5,0}
\definecolor{webbrown}{rgb}{.6,0,0}
\usepackage[colorlinks=true,
linkcolor=webgreen,
filecolor=webbrown,
citecolor=webgreen]{hyperref}

\newcommand{\seqnum}[1]{\href{http://oeis.org/#1}{\underline{#1}}}
\newcommand{\Figs}[1]{\hyperref[#1]{Figure~\ref*{#1}}}
\newcommand{\Tabs}[1]{\hyperref[#1]{Table~\ref*{#1}}}
\newcommand{\Sec}[1]{\hyperref[#1]{section~\ref*{#1}}}
\newcommand{\Prop}[1]{\hyperref[#1]{Proposition~\ref*{#1}}}

\everymath{\displaystyle}

\theoremstyle{plain}
\newtheorem{theorem}{Theorem}
\newtheorem{lemma}[theorem]{Lemma}

\newtheorem{proposition}[theorem]{Proposition}

\theoremstyle{remark}
\newtheorem{remark}[theorem]{Remark}

\title{\bf  A generating polynomial for the pretzel knot}
\author{Franck Ramaharo\\
\small D\'epartement de Math\'ematiques et Informatique\\[-0.8ex]
\small Universit\'e d'Antananarivo\\[-0.8ex] 
\small 101 Antananarivo, Madagascar\\
\small\href{mailto:franck.ramaharo@gmail.com}{\tt franck.ramaharo@gmail.com}\\
}

\date{\small\today\\}

\begin{document}

\maketitle

\begin{abstract}
We collect statistics which consist of the coefficients in the expansion of the generating polynomials that count the Kauffman states associated with certain classes of  pretzel knots having $ n $ tangles, of $ r $ half-twists respectively. 

\end{abstract}

\section{Introduction}
The \textit{generating polynomial} for the shadow diagram of  the knot $ K $  provides a refinement of counting the corresponding Kauffman states \cite{Kauffman}. By state  is meant the diagram obtained by \textit{splitting} each vertex representing  the initial diagram crossings, i.e., each  \protect\includegraphics[width=.03\linewidth,valign=c]{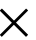} to \protect\includegraphics[width=.03\linewidth,valign=c]{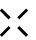},  and repasting the edges as either \protect\includegraphics[width=.03\linewidth,valign=c]{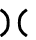} or \protect\includegraphics[width=.035\linewidth,valign=c]{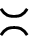}. The generating polynomial for the  knot $ K $ is then defined as the summation  which is taken over all its states, namely
\begin{equation}\label{eq:genpoly}
K(x)=\sum_{S}^{} x^{|S|},
\end{equation}
with  $ |S| $ denoting the number of Jordan curves in the state diagram $ S $. For instance, the generating polynomial for the Hopf link is $ L(x)=2x^2+2x $ (see \Figs{Fig:1LinkStates}).

\begin{figure}[ht]
\centering
\includegraphics[width=.4\linewidth]{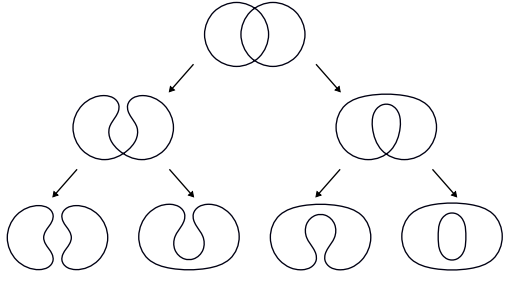}
\caption{The states of the Hopf link.}
\label{Fig:1LinkStates}
\end{figure}

If $ K $ and $ K' $ are two arbitrary diagrams, and $ \bigcirc $ denotes the unknot diagram, then we have the following properties and notations \cite{Ramaharo}:
\begin{enumerate}[(i)]
\item $ \bigcirc(x)=x $;
\item $ (\underbrace{\bigcirc\sqcup\bigcirc\sqcup\cdots\sqcup\bigcirc}_{n\ copies})(x):=\left(\bigcirc\bigcirc\cdots\bigcirc\right)(x)=x^n $;
\item $ \left(K\sqcup\bigcirc\right)(x)=xK(x) $;
\item $ \left(K\sqcup K'\right)(x)=K(x).K'(x) $;
\item $ \left(K\#\bigcirc\right) (x)=K(x)$;
\item $ K_n(x):=(\underbrace{K\#K\#\cdots\#K}_{n\ copies})(x)=x\left(x^{-1}K(x)\right)^n $, with $ K_0=\bigcirc $;\label{item:Dn}
\item $ \left(K\#K'\right)(x)=x^{-1}K(x).K'(x) $, 
\end{enumerate}
where $ \# $ and $ \sqcup $ are respectively the connected sum and the disjoint union. Moreover, if we let $ \overline{K} $ denote the closure of $ K $, i.e., the connected sum with itself, then there exist two polynomials $ \upalpha,\upbeta\in\mathbb{N}[x]$  such that 
\begin{equation}\label{eq:alphabeta}
\overline{K}(x)=x^2\upalpha(x)+x\upbeta(x),\
\mbox{with}\
{K}(x)=x^2\upbeta(x)   +x\upalpha(x).
\end{equation}

With the notation and property in \eqref{item:Dn} we obtain $ \overline{K_0}(x)=\left(\bigcirc\bigcirc\right)(x)=x^2 $ and
\begin{align}
\overline{K_n}(x)&=\upalpha(x)\overline{K_{n-1}}(x)+\upbeta(x)K_{n-1}(x)\label{eq:gp1}\\
&=\big(\upalpha(x)+x\upbeta(x)\big)^n+\left(x^2-1\right)\upalpha(x)^n\label{eq:gp2}.
\end{align}

We can interpret \eqref{eq:alphabeta} as follows: given the closure \protect\includegraphics[width=.07\linewidth,valign=c]{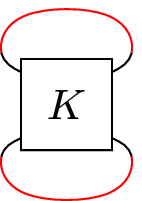} of a knot diagram \protect\includegraphics[width=.09\linewidth,valign=c]{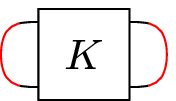}, its  state diagrams  can be divided into two subsets that are respectively counted by $ x^2\upalpha(x) $ and $ x\upbeta(x) $ as represented in \Figs{Fig:split}. 

\begin{figure}[H]
\centering
\hspace*{\fill}
\subfigure[States counted by $ x^2\upalpha(x) $.]{\includegraphics[width=0.28\linewidth]{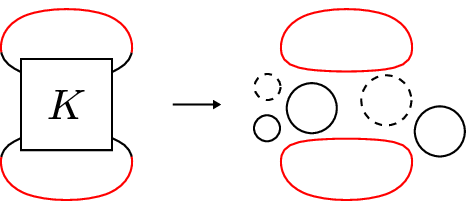}}\hfill%
\subfigure[States counted by $ x\upbeta(x) $.]{\includegraphics[width=0.28\linewidth]{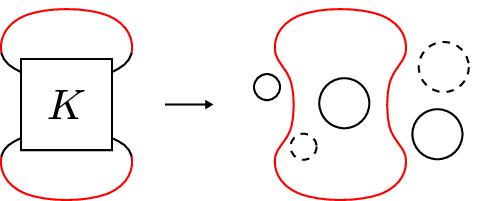}}
\hspace*{\fill}
\caption{The two subsets of states associated with the closure of a knot.}
\label{Fig:split}
\end{figure}

In this note, we shall take advantage of these properties and  establish the generating polynomial for a particular class of the pretzel knots.

\section{Pretzel knot}\label{sec:pretzel}
A \textit{pretzel knot} $ P_{n,r}:=P(r,r,\ldots,r) $ \cite{Stoimenow} is a knot composed of $ n $ pairs of strands twisted $  r $ times and attached along the tops and bottoms as in \Figs{Fig:PretzelLink}\subref{subfig:pretzel}.

\begin{figure}[ht]
\centering
\hspace*{\fill}
\subfigure[$P(r,r,\ldots,r)$]{\includegraphics[width=0.3\linewidth]{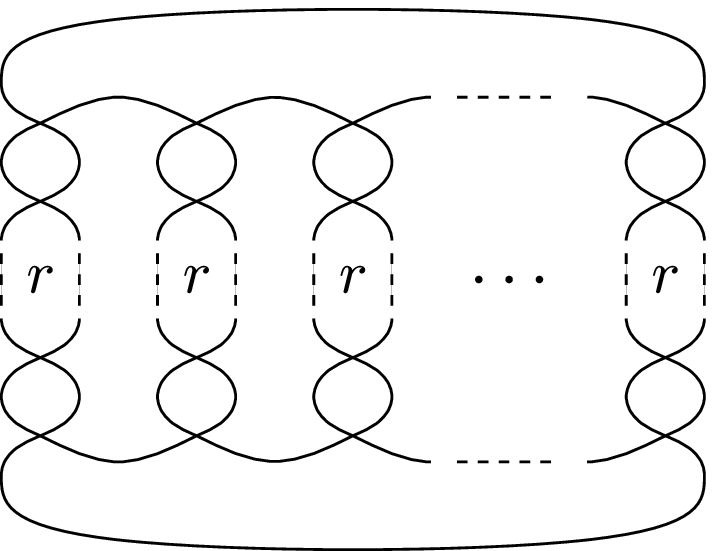}\label{subfig:pretzel}}\hfill%
\subfigure[$P(r,r,\ldots,r):= \overline{F_r\#F_r\#\cdots\#F_r}$]{\includegraphics[width=0.5875\linewidth]{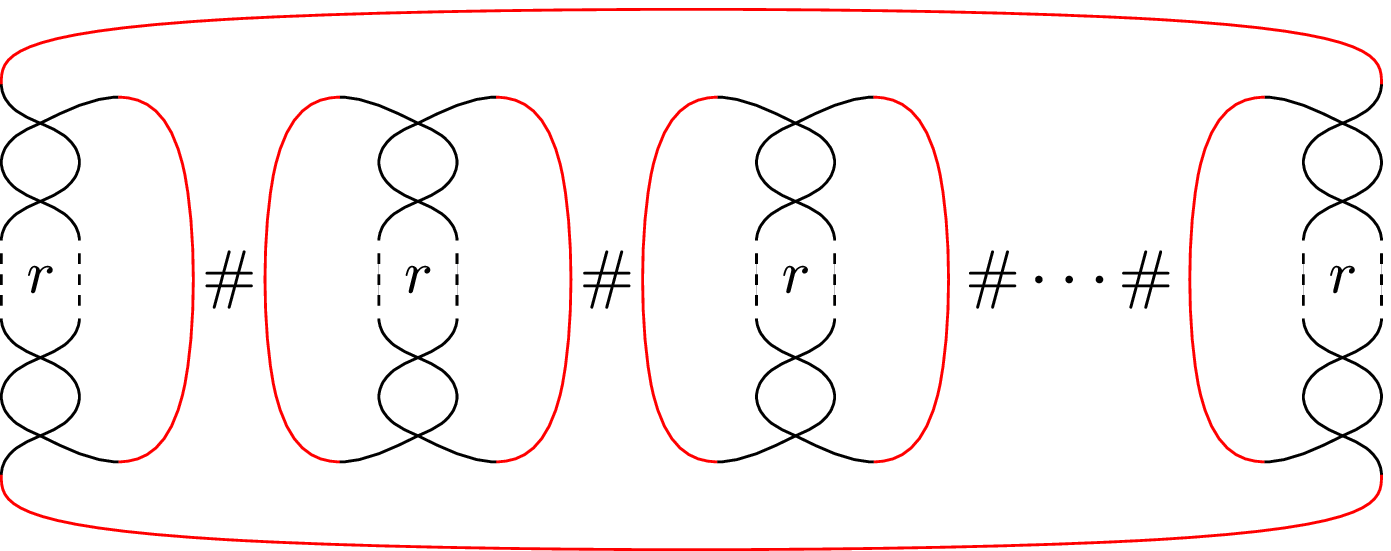}\label{subfig:Fr}}
\hspace*{\fill}
\caption{The shadow diagram for the pretzel knot  and the corresponding connected sums for constructing it.}
\label{Fig:PretzelLink}
\end{figure}

\begin{figure}[ht]
\centering
\hspace*{\fill}
\subfigure[$P_{n,0}$]{\includegraphics[width=0.25\linewidth]{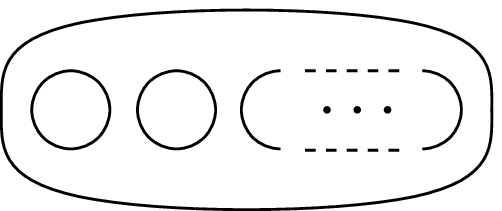}\label{subfig:P0n}}\hfill %
\subfigure[$P_{n,1}$]{\includegraphics[width=0.2\linewidth]{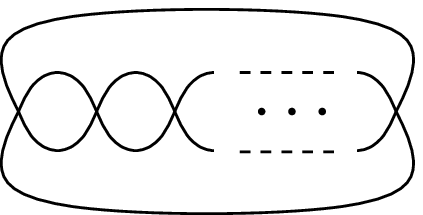}\label{subfig:P1n}}\hfill%
\subfigure[$P_{n,2} $]{\includegraphics[width=0.35\linewidth]{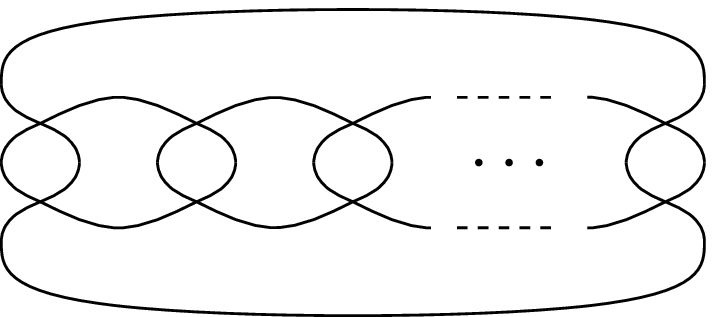}\label{subfig:P2n}}\hfill %
\hspace*{\fill}
\\
\hspace*{\fill}
\subfigure[$ P_{n,3}$]{\includegraphics[width=0.3\linewidth]{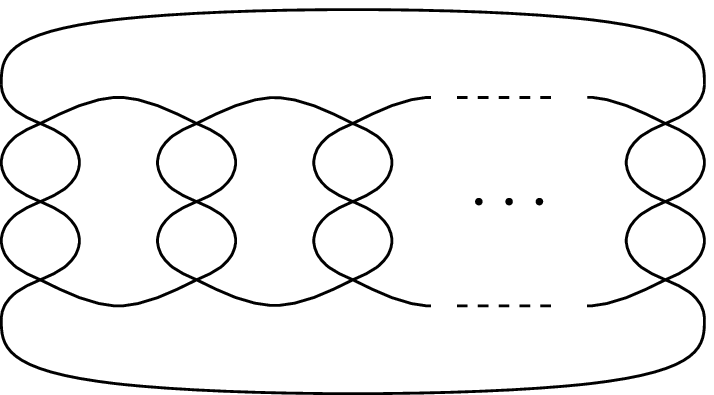}\label{subfig:P3n}}\hfill%
\subfigure[$ P_{1,r}$]{\includegraphics[width=0.1125\linewidth]{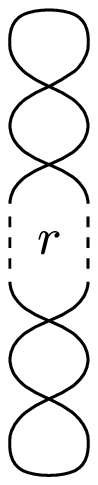}\label{subfig:Pr1}}\hfill%
\subfigure[$P_{2,r}$]{\includegraphics[width=0.11\linewidth]{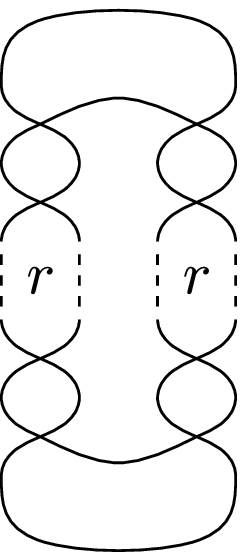}\label{subfig:Pr2}}\hfill%
\hspace*{\fill}
\caption{For some values of $ n $ and $ r $ we have: \subref{subfig:P0n} a disjoint union of $ n $ unknots ($ r=0 $, $ n\geq 1 $); \subref{subfig:P1n}  an \textit{$ n $-foil} ($ r=1 $); \subref{subfig:P1n}  an \textit{$ n $-chain link} ($ r=2 $); \subref{subfig:P2n} an \textit{$n $-sinnet of square knotting} $ (r=3) $; \subref{subfig:Pr1} an \textit{$ r $-twist loop} ($ n=1 $) and \subref{subfig:Pr2} a $ 2r $-foil ($ n=2 $).}
\label{Fig:somePnr}
\end{figure}

If $ F_r $ denotes the \textit{$ r $-foil} as pictured in \Figs{Fig:PretzelLink}\subref{subfig:Fr}, then we have $ P_{n,r}:=\overline{\left(F_r\right)_{n}} $. For the convenience, we set $ F_0=\bigcirc\bigcirc $ and $ \left(F_r\right)_0=\bigcirc $  so that $ \overline{\left(F_{0}\right)_0}=\bigcirc \bigcirc $ and $ \overline{\left(F_{0}\right)_n}=\protect\includegraphics[width=.15\linewidth,valign=c]{0r-pretzel} $ for $ n\geq 1 $. \Figs{Fig:somePnr} displays some pretzel knots for small values of $ n $ and $ r $.

\section{Generating polynomial}
We begin with the generating polynomial for the closure of the $ r $-foil  (see \Figs{Fig:PretzelLink}\subref{subfig:Fr}) which yields the $ r $-twist loop  (see \Figs{Fig:somePnr}\subref{subfig:Pr1}).

\begin{lemma}[\cite{Ramaharo}]
The generating polynomials for the $ r $-twist loop and the $ r $-foil knot are respectively given by 
\begin{equation}\label{eq:Tnx2}
T_r(x)=x(x+1)^r
\end{equation}
and
\begin{equation}\label{eq:Fnx2}
F_r(x)=\overline{T_r}(x)=(x+1)^r+x^2-1.
\end{equation}
\end{lemma}

We shall deduce the two polynomials $ \upalpha_r $, $ \upbeta_r $ associated with closure of the $ r $-foil with the help of formula \eqref{eq:Tnx2}.

\begin{lemma}
The following expression holds for $ \overline{F_r}(x) $
\begin{equation}\label{eq:alphadef}
\overline{F_r}(x)=x^2\upalpha_r(x)+x \upbeta_{r}(x),
\end{equation}
where $\upalpha_r(x):=\dfrac{(x+1)^r-1}{x} $ and $ \upbeta_{r}(x)=1 $.
\end{lemma}
\begin{proof}
First, note that $ \overline{F_r}=\overline{\left(F_r\right)_1}=P_{1,r}=T_r $. Among the states of the $ r $-twist loop,  there is exactly one which is made up of one component as shown in 	\Figs{Fig:alphabeta}. Hence $ \upbeta_r(x)=1 $.
Then by \eqref{eq:Tnx2}, we get
\begin{equation*}
T_r(x)=x^2\left(\dfrac{(x+1)^r-1}{x}\right)+x.
\end{equation*}

In fact if we let  $ \upalpha_r(x)=\dfrac{(x+1)^r-1}{x} $, then the expansion of $ x^2\upalpha_r(x) $, namely
\begin{align*}
x^2\upalpha_r(x)&=x\Big(x(x+1)^0+x(x+1)^1+x(x+1)^2+\cdots+x(x+1)^{r-1}\Big)\\
&=\left(\bigcirc\sqcup\bigcirc\right)(x)+\left(\bigcirc\sqcup\protect\includegraphics[width=.045\linewidth,valign=c]{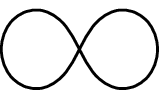}\right)(x)+\left(\bigcirc\sqcup\protect\includegraphics[width=.074\linewidth,valign=c]{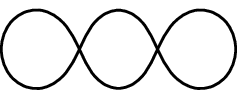}\right)(x)+\cdots+\left(\bigcirc\sqcup\protect\includegraphics[width=.11\linewidth,valign=c]{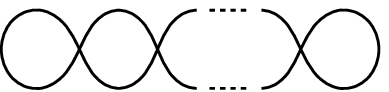}\right)(x),
\end{align*}
counts as expected the states which might result to that in \Figs{Fig:alphabeta}\ref{subfig:alpha1}.
\begin{figure}[ht]
\centering
\hspace*{\fill}
\subfigure[$2^r-1 $ states of the kind.]{\includegraphics[width=0.275\linewidth]{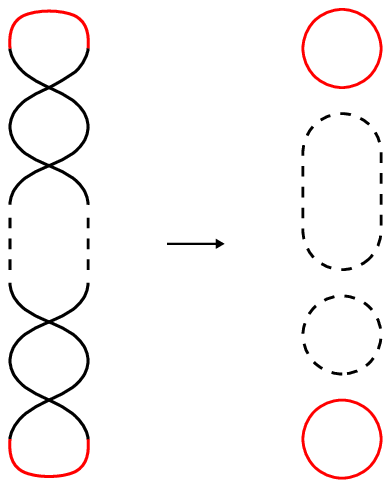}\label{subfig:alpha1}}\hfill%
\subfigure[One state of the kind.]{\includegraphics[width=0.275\linewidth]{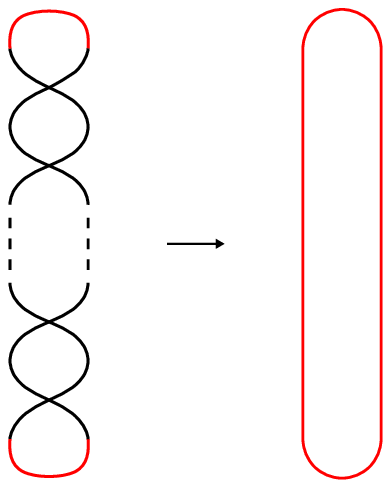}}
\hspace*{\fill}
\caption{Representatives of the $ r $-twist loop states.}
\label{Fig:alphabeta}
\end{figure}
\end{proof}

\begin{proposition}\label{prop:gp}
The generating polynomial for the Pretzel knot $ P_{n,r} $ is given by
\begin{align*}
{P}_{n,r}(x)&=\big(\upalpha_r(x)+x\big)^n+\left(x^2-1\right)\upalpha_r(x)^n.
\end{align*}
\end{proposition}

\begin{proof}
We write
\begin{equation*}
P_{n,r}(x)=\overline{\left(F_r\right)_{n}}(x) =\upalpha_{r}(x)P_{n-1,r}(x)+\left(F_r\right)_{n-1}(x).
\end{equation*}
and we conclude by \eqref{eq:gp2}.
\end{proof}

\begin{remark}
Since $\left(F_r\right)_n(x)= x\left(\upalpha_r(x)+x\right)^n $, and for some values of $ r $, we obtain the generating polynomials for the following knots \cite{Ramaharo}:
\begin{itemize}
\item  $\left(F_1\right)_n= x(x+1)^n  $, $ n $-twist loop \cite[\seqnum{A097805}, \seqnum{A007318}]{Sloane};
\item $\left(F_2\right)_n =  x(2x+2)^n $, \textit{$ n $-link} \cite[\seqnum{A038208}]{Sloane};
\item  $\left(F_3\right)_n =x(x^2+4x+3)  $ ,  \textit{$ n $-overhand knot} \cite[\seqnum{A299989}]{Sloane}.
\end{itemize}

\end{remark}
\section{Results}
In this section, we retrieve some of our previous results (case $ r=1,2,3 $) \cite{Ramaharo} which confirm  that the generating polynomial agrees with the construction in \Sec{sec:pretzel}.

\begin{enumerate}
\item  \textbf{Case {\boldmath $ r=0$}}.
\begin{enumerate}
\item Generating polynomial:
\begin{equation*}
P_{n,0}(x)=
\begin{cases}
x^2& \textit{if $ n=0 $};\\
x^n&\textit{if $  n\geq 1 $}.
\end{cases}
\end{equation*}
\item Coefficients table: \cite[\seqnum{A010054}, \seqnum{A023531}, \seqnum{A073424}  ($ n\geq 1 $,  all read as triangle)]{Sloane}
\begin{table}[H]
\centering
$\begin{array}{c|rrrrrrrrrrrrrrrr}
n\ \backslash\ k	&0	&1	&2	&3	&4	&5	&6&7\\
\midrule
0	&0	&0	&1	&	&	&	&	&			\\
1	&0	&1	&	&	&	&	&	&			\\
2	&0	&0	&1	&	&	&	&	&			\\
3	&0	&0	&0	&1	&	&	&	&			\\
4	&0	&0	&0	&0	&1	&	&	&			\\
5	&0	&0	&0	&0	&0	&1	&	&			\\
6	&0	&0	&0	&0	&0	&0	&1	&			\\
7	&0	&0	&0	&0	&0	&0	&0	&1		\\	
\end{array}$	
\caption{Values of $ p_0(n,k) $ for  $ 0\leq n\leq 8 $ and $ 0\leq k\leq 8 $.}
\label{Tab:0101}
\end{table}
\end{enumerate}

\item  \textbf{Case {\boldmath $ r=1 $}}.
\begin{enumerate}
\item Generating polynomial:
\begin{align*}
{P}_{n,1}(x)=(x+1)^n+x^2-1.
\end{align*}
\item Coefficients table: \cite[\seqnum{A007318} ($ 3\leq k\leq  n $ )]{Ramaharo, Sloane}
\begin{table}[H]
\centering
$\begin{array}{c|rrrrrrrrrrrrr}
n\ \backslash\ k	&0	&1	&2	&3	&4	&5	&6	&7	&8 &9 &10	\\
\midrule
0	&0	&0	&1	&	&	&	&	&	&		\\
1	&0	&1	&1	&	&	&	&	&	&		\\
2	&0	&2	&2	&	&	&	&	&	&		\\
3	&0	&3	&4	&1	&	&	&	&	&		\\
4	&0	&4	&7	&4	&1	&	&	&	&		\\
5	&0	&5	&11	&10	&5	&1	&	&	&		\\
6	&0	&6	&16	&20	&15	&6	&1	&	&		\\
7	&0	&7	&22	&35	&35	&21	&7	&1	&	\\
8	&0	&8	&29	&56	&70	&56	&28	&8	&1	\\
9	&0	&9	&37	&84	&126	&126	&84	&36	&9	&1	&	\\
10	&0	&10	&46	&120	&210	&252	&210	&120	&45	&10	&1\\
\end{array}$
\caption{Values of $ p_1(n,k)$ for $ 0\leq n\leq 10 $ and $ 0\leq k\leq 10$.}
\label{Tab:gpfoil}
\end{table}
\end{enumerate}

\item  \textbf{Case {\boldmath $ r=2 $}}.
\begin{enumerate}
\item Generating polynomial:
\begin{align*}
{P}_{n,2}(x)=\left(2x+2\right)^n+\left(x^2-1\right)\left(x+2\right)^n.
\end{align*}
\item Coefficients table: \cite[\seqnum{A300184}]{Sloane}
\begin{table}[H]
\centering
$\begin{array}{c|rrrrrrrrrrrr}
n\ \backslash\ k		 &0		 &1		 &2		 &3		 &4		 &5		 &6		 &7		 &8		 &9		 &10		 &11\\
\midrule
0	 &0	 &0	 	 &1	 &	 &	 &	 &	 &	 &	 &	 &	 &\\
1	 &0	 &1		 &2		 &1	 &	 &	 &	 &	 &	 &	 &	 &\\
2	 &0	 &4	 	 &7	  	 &4		 &1	 &	 &	 &	 &	 &	 &	 &\\
3	 &0	 &12	 &26	 &19	 &6		 &1	 &	 &	 &	 &	 &	 &\\
4	 &0	 &32	 &88	 &88	 &39	 &8		 &1	 &	 &	 &	 &	 &\\
5	 &0	 &80	 &272	 &360	 &1230	 &71	 &10	 &1	 &	 &	 &	 &\\
6	 &0	 &192	 &784	 &1312	 &1140	 &532	 &123	 &12	 &1	 &	 &	 &\\
7	 &0	 &448	 &2144	 &4368	 &4872	 &3164	 &1162	 &211	 &14	 &1	 &	 &\\
8	 &0	 &1024	 &5632	 &13568	 &18592	 &15680	 &8176	 &2480	 &367	 &16	 &1	 &\\
9	 &0	 &2304	 &14336	 &39936	 &65088	 &67872	 &46368	 &20304	 &5262	 &655	 &18	 &1
\end{array}$
\caption{Values of $ p_2(n,k) $ for $ 0\leq n\leq 9 $ and $ 0\leq k\leq 11 $.}
\label{tab:nlinkchain}
\end{table}
\end{enumerate}

\item  \textbf{Case {\boldmath $ r=3 $}}.
\begin{enumerate}
\item Generating polynomial:
\begin{align*}
{P}_{n,3}(x)=\left(x^2+4x+3\right)^n+\left(x^2-1\right)\left(x^2+3x+3\right)^n.
\end{align*}
\item Coefficients table: \cite[Table 14]{Ramaharo}
\begin{table}[H]
\centering
$\begin{array}{c|rrrrrrrrrrrrrrrr}
n\ \backslash\ k		 &0		 &1		 &2		 &3		 &4		 &5		 &6		 &7		 &8		 &9		 &10		 &11		 &12 \\
\midrule
0 &0 & 0 & 1\\
1 & 0 & 1 & 3 & 3 & 1\\
2 & 0 & 6 & 16 & 20 & 15 & 6 & 1\\
3 & 0 & 27 & 90 & 136 & 129 & 84 & 36 & 9 & 1\\
4 & 0 & 108 & 459 & 876 & 1021 & 832 & 501 & 220 & 66 & 12 & 1\\
5 & 0 & 405 & 2133 & 5085 & 7350 & 7321 & 5420 & 3103 & 1375 & 455 & 105 & 15 & 1\\
\end{array} $ 
\caption{Values of $ p_3(n,k) $ for $ 0\leq n\leq 5 $ and $ 0\leq k\leq 12 $.}
\label{tab:sinnet}
\end{table}
\end{enumerate}

\item  \textbf{Case {\boldmath $ r=n $}}.
\begin{enumerate}
\item Generating polynomial:
\begin{equation*}
P_{n,n}(x)=\left(\dfrac{(x+1)^n-1}{x}+x\right)^n+\left(x^2-1\right)\left(\dfrac{(x+1)^n-1}{x}\right)^n.
\end{equation*}
\item Coefficients table: 
\begin{table}[H]
\centering
\resizebox{\linewidth}{!}{%
$\begin{array}{c|rrrrrrrrrrrrrrrrr}
n\ \backslash\ k		 &0		 &1		 &2		 &3		 &4		 &5		 &6		 &7		 &8		 &9		 &10		 &11		 &12 	&13	&14\\
\midrule
0 & 0 & 0 & 1\\
1 & 0 & 1 & 1\\
2 & 0 & 4 & 7 & 4 & 1\\
3 & 0 & 27 & 90 & 136 & 129 & 84 & 36 & 9 & 1\\
4 & 0 & 256 & 1504 & 4336 & 8273 & 11744 & 13036 & 11488 & 8014 & 4368 & 1820 & 560 & 120 & 16 & 1\\
\end{array}$}
\caption{Values of $ p_n(n,k) $ for $ 0\leq n\leq 4 $ and $ 0\leq k\leq 14 $.}
\label{tab:Pnn}
\end{table}
\end{enumerate}

\item \textbf{Case {\boldmath $ k=1 $}}: $ p_r(n,1)=nr^{n-1},\ r\geq 1$ \cite[\seqnum{A104002} $ (1\leq r\leq n  ) $]{Sloane}, see \Tabs{tab:k1}.
\begin{table}[ht]
\centering
$\begin{array}{c|rrrrrrrrrrrrrrrrrrrrrrrrr}
n\ \backslash\ r		 &0		 &1		 &2		 &3		 &4		 &5		 &6		 &7	&8&9\\
\midrule
0 & 0 & 0 &     0&       0&        0&         0&          0&          0&           0&           0\\
1 & 1 & 1 &     1&       1&        1&         1&          1&          1&           1&           1\\
2 & 0 & 2 &     4&       6&        8&        10&         12&         14&          16&          18\\  
3 & 0 & 3 &    12&      27&       48&        75&        108&        147&         192&         243\\ 
4 & 0 & 4 &    32&     108&      256&       500&        864&       1372&        2048&        2916\\  
5 & 0 & 5 &    80&     405&     1280&      3125&       6480&      12005&       20480&       32805\\ 
6 & 0 & 6 &   192&    1458&     6144&     18750&      46656&     100842&      196608&      354294\\
7 & 0 & 7 &   448&    5103&    28672&    109375&     326592&     823543&     1835008&     3720087\\
8 & 0 & 8 &  1024&   17496&   131072&    625000&    2239488&    6588344&    16777216&    38263752\\
\end{array}$ 
\caption{Values of $ p_r(n,1) $ for $ 0\leq n\leq 8 $ and $ 0\leq r\leq 9 $ .}
\label{tab:k1}
\end{table}
\item \textbf{Case {\boldmath $ k=2 $}}: $ p_r(n,2)=\binom{n}{2}r^{n-2}\left(2\binom{r}{2}+1\right)+r^n,\ r\geq 1$, see \Tabs{tab:k2}.

\begin{table}[ht]
\centering
\resizebox{\linewidth}{!}{%
$\begin{array}{c|rrrrrrrrrrrrrrrrrrrrrrrrr}
n\ \backslash\ r		 &0		 &1		 &2		 &3		 &4		 &5		 &6		 &7		&8\\
\midrule
0	&1 &  1  &    1 &       1 &        1 &         1 &          1 &           1 &           1 \\ 
1	&0 &  1  &    2 &       3 &        4 &         5 &          6 &           7 &           8 \\ 
2	&1 &  2  &    7 &      16 &       29 &        46 &         67 &          92 &         121 \\ 
3	&0 &  4  &   26 &      90 &      220 &       440 &        774 &        1246 &        1880 \\ 
4	&0 &  7  &   88 &     459 &     1504 &      3775 &       7992 &       15043 &       25984 \\  
5	&0 & 11  &  272 &    2133 &     9344 &     29375 &      74736 &      164297 &      324608 \\
6	&0 & 16  &  784 &    9234 &    54016 &    212500 &     649296 &     1666294 &     3764224 \\ 
7	&0 & 22  & 2144 &   37908 &   295936 &   1456250 &    5342112 &    16000264 &    41320448 \\
8	&0 & 29  & 5632 &  149445 &  1556480 &   9578125 &   42177024 &   147414197 &   435159040 \\
9	&0 & 37  &14336 &  570807 &  7929856 &  61015625 &  322486272 &  1315198171 &  4437573632 \\
\end{array}$}
\caption{Values of $ p_r(n,2) $ for $ 0\leq n\leq 9 $ and $ 0\leq r\leq 8 $.}
\label{tab:k2}
\end{table}
\end{enumerate}

We observe the following formulas:
\begin{itemize}
\item First column in \Tabs{tab:k2} is $ 1,0,1 $ followed by $ 0,0,0,\ldots $ \cite[\seqnum{A154272}]{Sloane};
\item $ p_1(n,2)=\binom{n}{2}+1 $ \cite[\seqnum{A152947}]{Sloane};
\item $ p_0(2,0)=0$, $p_1(2,1)=2$, $p_2(2,2)=7$  and $p_n(2,n)=\binom{2n}{n} $ \cite[\seqnum{A000984}]{Sloane};
\item $ p_2(n,2)= (3n^2 - 3n + 8)2^{n-3} $ \cite[\seqnum{A300451}]{Sloane};
\item $ p_2(n,n)= 2n(n-1)+2^n-1 $ \cite[\seqnum{A295077}]{Sloane};
\item $ p_3(n,2n-1)=\binom{3n}{3} $ \cite[\seqnum{A006566}]{Sloane};
\item $ p_3(n,2n)=\binom{3n}{2} $ \cite[\seqnum{A062741}]{Sloane};
\item $ p_0(0,1)=0 $ and $ p_n(n,1)=n^n $ \cite[\seqnum{A000312}]{Sloane};
\item $ p_n(2,2)=2n^2 - n + 1$ \cite[\seqnum{A130883}]{Sloane};
\item $ p_n(n,n^2-n+1)=n^2 $ \cite[\seqnum{A000290}]{Sloane};
\item $ p_0(0,0)=p_1(1,0)=0 $, $ p_2(2,2)=7 $ and $ p_n(n,n^2-n)=\binom{n^2}{2} $ \cite[\seqnum{A083374}]{Sloane};
\item $ p_n(n,n^2-n-1)=\binom{n^2}{3} $ \cite[\seqnum{A178208}]{Sloane}.
\end{itemize}

\small \textbf{2010 Mathematics Subject Classifications}:  05A19;  57M25.

\begin{thebibliography}{99}

\bibitem{Kauffman} 
Louis H. Kauffman, State models and the Jones polynomial,  {\em Topology} {\bf 26} (1987), 95--107.

\bibitem{Ramaharo}
Franck Ramaharo, Statistics on some classes of knot shadows, arXiv preprint, \url{https://arxiv.org/abs/1802.07701}, 2018.

\bibitem{Sloane} 
Neil J. A. Sloane, 
{\em The On-Line Encyclopedia of Integer Sequences}, published electronically at \url{http://oeis.org}, 2018.

\bibitem{Stoimenow} 
Alexander Stoimenow, Everywhere equivalent $ 2 $-component links,  {\em Symmetry} {\bf 7} (2015), 365--375.
\end{thebibliography}
\end{document}